\documentclass[12pt]{article}
\usepackage[english]{babel}
\usepackage{amsmath,amssymb,dsfont, float, bm}
\usepackage{amsmath,amssymb,amsfonts,amsthm,mathtools,mathrsfs,booktabs}
\usepackage{tikz,float}
\usepackage{pgfplots}
\pgfplotsset{compat=1.16}
\usepackage{authblk}

\usepackage{natbib}

\usepackage[colorlinks=true, citecolor=blue, urlcolor=blue]{hyperref}
\usepackage{multirow}
\usepackage{bookmark}
\usepackage{booktabs}
\hypersetup{                 
colorlinks=true,
linkcolor=blue,
citecolor=blue,
anchorcolor=blue,
urlcolor=blue,
linktoc=page,
}
\usepackage[a4paper,margin=2.5cm]{geometry}

\newtheorem{theorem}{Theorem}

\newtheorem{corollary}{Corollary}[section]
\newtheorem{example}{Example}[section]

\newtheorem{proposition}{Proposition}[section]
\newtheorem{remark}{Remark}

\begin{document}

\title{\bf 
%
Characterizations of independence between order statistics and rank indicators
}

 \author[]{Roberto Vila }
 \author[]{\, Frederico Almeida \thanks{Corresponding author: frederico.almeida@unb.br 

 \qquad\qquad\qquad\qquad\qquad \ \ \ rovig161@gmail.com}}
 \affil[]{Department of Statistics, University of
 	Bras\'ilia, 70910-900, Bras\'ilia, Brazil}
\setcounter{Maxaffil}{0}
\renewcommand\Affilfont{\itshape\small}

\maketitle
\begin{abstract}
We study the independence between an order statistic and its corresponding rank indicator for independent nonnegative random variables. A general characterization is established through a probability measure obtained by reweighting the distribution of a single observation according to its conditional probability of occupying a prescribed rank. This framework yields explicit expressions for the associated weighting functions and conditional distributions, as well as distribution-free measures of departure from independence based on Kolmogorov and Wasserstein distances. Special attention is devoted to the minimum and maximum order statistics, leading to new characterizations of independent right- and left-censoring under both single and multiple censoring mechanisms. We also establish sufficient conditions based on proportional hazards and proportional reversed hazards models and derive complete characterizations in the classical single-censoring setting. Discrete and continuous examples, together with numerical illustrations, are presented to demonstrate the applicability of the proposed methodology.
\end{abstract}


\section{Introduction}\label{sec1}

{
The study of independence between order statistics and the mechanisms that determine their ranks plays a central role in probability theory and survival analysis \citep{arnold2008order,david2003order,balakrishnan1998order}. In many applications, the quantity of interest is an order statistic, such as the minimum or the maximum of a collection of random variables, together with an indicator identifying the observation responsible for that value. Understanding when these two quantities are independent provides valuable insight into the probabilistic structure of the underlying model and leads to tractable analytical representations.

From a probabilistic perspective, this problem can be interpreted as a particular instance of the more general question of determining when an order statistic is independent of the indicator identifying the observation responsible for that order \citep{david2003order,arnold2008order}. For right-censoring, this corresponds to the minimum order statistic, whereas left-censoring is naturally associated with the maximum order statistic.

In survival analysis, it is commonly assumed that the observed time $Z$ and the event indicator $\delta$ are independent, where $Z=\min\left\{Y, C\right\}$ (for right-censoring) or $Z=\max\left\{Y, C\right\}$ (for left-censoring), and $\delta=\mathds{1}_{\{Y \leqslant C\}}$ (or $\delta=\mathds{1}_{\{Y \geqslant C\}}$). Here, $Y$ and $C$ denote the failure and censoring times, respectively. This assumption, often referred to as independent (or non-informative) censoring, plays a fundamental role in survival analysis: it greatly simplifies the estimation of survival functions and enables valid, unbiased statistical inference under both right- and left-censoring mechanisms \citep{kaplan1958nonparametric,cox1972regression,aalen1978nonparametric,fleming1991counting,cox1975partial,kalbfleisch2002statistical,klein2006survival}.

Although independence between the observed time $Z$ and the event indicator $\delta$ is not a strict requirement in classical survival analysis theory, its presence is a probabilistic property of considerable interest. When this condition holds, the joint distribution of $(Z, \delta)$ simplifies substantially, leading to more tractable analytical expressions, more elegant theoretical developments, and potentially more efficient inferential procedures. Moreover, independence ensures that the distribution of observed times is the same for individuals who experience the event of interest and for those who are censored, eliminating any influence of the event indicator on the distribution of $Z$. These features make the independence between $Z$ and $\delta$ relevant from both theoretical and methodological perspectives, motivating the study of conditions under which this property can be guaranteed \citep{cox1975partial,kalbfleisch2002statistical,klein2006survival}. Several important classes of models have been proposed to study these conditions, among which proportional hazards and proportional reversed hazards models occupy a particularly prominent position \citep{cox1972regression,gupta1998proportional}.

Despite its importance, the independence between $Z$ and $\delta$ is not always straightforward, especially in complex probabilistic models or scenarios where various data-generating mechanisms affect their dependence structure. In this context, identifying the distributions and parameter configurations under which $Z$ and $\delta$ are independent, as well as developing rigorous metrics to quantify deviations from this property, represents a significant contribution to the theory of survival analysis. Therefore, it is essential to devise tools that enable a quantitative assessment of the degree of independence. Such investigations yield deeper insights into how the censoring mechanism influences the distribution of observed survival times.


The main contributions of this work are as follows. First, we establish a general characterization of the independence between an order statistic and its corresponding rank indicator through a probability measure obtained by reweighting the distribution of a single observation according to its conditional probability of occupying a prescribed rank. Second, we derive explicit expressions for the associated weighting functions, conditional distributions, survival functions, and density functions. Third, we specialize the general framework to multiple right- and left-censoring schemes, obtaining necessary and sufficient conditions for independence. Fourth, we establish broad sufficient conditions based on proportional hazards and proportional reversed hazards models, together with complete characterizations in the classical single-censoring setting. Finally, we introduce distribution-free measures based on Kolmogorov and Wasserstein distances, complemented by graphical tools and numerical illustrations that quantify departures from independence. Such distances have become standard tools for comparing probability distributions and measuring discrepancies between stochastic models \citep{shorack1986empirical,villani2009optimal}.


The remainder of this paper is organized as follows. Section~\ref{sec2} introduces the general probabilistic framework and establishes the main characterization result. Sections~\ref{sec3} and~\ref{sec4} apply this framework to multiple right- and left-censoring mechanisms, respectively, deriving measures of departure from independence and characterizations based on proportional hazards and proportional reversed hazards models. Section~\ref{sec7} provides numerical illustrations, and Section~\ref{FinalR} concludes the paper.



}


\section{A general characterization}\label{sec2}

This section develops the general framework underlying the paper. We introduce a probability measure obtained by reweighting the distribution of one observation according to its conditional probability of occupying a prescribed rank and characterize the independence between an order statistic and its corresponding rank indicator in terms of this measure. Explicit expressions for the associated weighting function and the survival function of the order statistic are also derived.

{
\begin{theorem}\label{thm:kth-order-general}
	Let \(X_1,\ldots,X_m\) be independent nonnegative random variables, let
	\[
	Z=X_{k:m},
	\quad
	1\leqslant k\leqslant m,
	\]
 be the \(k\)-th order statistic and, for \(i=1,\ldots,m\), define the rank indicator, indicating whether $X_i=X_{k:m}$, by
	\[
	\delta_i
	=
	\mathds{1}_{\left\{
		\sum_{j\ne i}\mathds{1}_{\{X_j<X_i\}}<k
		\leqslant
		1+\sum_{j\ne i}\mathds{1}_{\{X_j\leqslant X_i\}}
		\right\}}.
	\]
	Furthermore, let
	\[
	W_{k,i}(x)
	=
	\mathbb P(\delta_i=1\mid X_i=x),
	\]
	assume that \(\mathbb E[W_{k,i}(X_i)]>0\), and define the probability measure
	\[
	\mathbb P_{k,i}^*(A)
	=
	\frac{
		\mathbb E\!\left[
		W_{k,i}(X_i)\mathds{1}_{\{X_i\in A\}}
		\right]
	}{
		\mathbb E\!\left[
		W_{k,i}(X_i)
		\right]
	},
	\quad
	A\in\mathcal B(\mathbb R_+),
	\]
	with survival function \(S_{k,i}^*(z)=1-\mathbb P_{k,i}^*((-\infty,z])\), where \(\mathcal{B}(\mathbb{R}_+)\) denotes the Borel \(\sigma\)-algebra on
\(\mathbb{R}_+=[0,\infty)\). We have
	\[
	S_{k,i}^*(z)=S_{k:m}(z),
	\quad
	z\geqslant 0,
	\]
	if and only if \(Z\) and \(\delta_i\) are independent.
\end{theorem}
\begin{proof}
For every Borel set \(A\in\mathcal{B}(\mathbb{R}_+)\), the definition of
conditional expectation gives
\begin{align*}
\mathbb E\!\left[
W_{k,i}(X_i)\mathds{1}_{\{X_i\in A\}}
\right]
&=
\mathbb E\!\left[
\mathbb E[\delta_i\mid X_i]\mathds{1}_{\{X_i\in A\}}
\right]  
\\[0,2cm]
&=
\mathbb E\!\left[
\delta_i\mathds{1}_{\{X_i\in A\}}
\right]
=
\mathbb P(X_i\in A,\delta_i=1).
\end{align*}
Likewise, taking $A=\mathbb{R}_+$, 
$
\mathbb E[W_{k,i}(X_i)]
=
\mathbb E[\delta_i]
=
\mathbb P(\delta_i=1).
$
Since \(\mathbb E[W_{k,i}(X_i)]>0\), it follows that
\[
\mathbb P_{k,i}^*(A)
=
\frac{\mathbb P(X_i\in A,\delta_i=1)}
{\mathbb P(\delta_i=1)}
=
\mathbb P(X_i\in A\mid\delta_i=1).
\]

Now, on the event \(\{\delta_i=1\}\), the random variable \(X_i\) occupies
the \(k\)-th position among \(X_1,\ldots,X_m\). Hence,
$X_i=Z$
on $\{\delta_i=1\}$,
and therefore
$
\mathbb P(X_i\in A,\delta_i=1)
=
\mathbb P(Z\in A,\delta_i=1).
$
Consequently,
\[
\mathbb P_{k,i}^*(A)
=
\mathbb P(Z\in A\mid\delta_i=1),
\]
showing that \(\mathbb P_{k,i}^*\) is precisely the conditional distribution
of \(Z\) given \(\delta_i=1\). In particular,
\[
S_{k,i}^*(z)
=
\mathbb P(Z>z\mid\delta_i=1),
\quad z\geqslant0.
\]

Suppose first that
$
S_{k,i}^*(z)=S_{k:m}(z),
\ z\geqslant 0.
$
Since
$
S_{k:m}(z)=S_Z(z),
$
we obtain
\[
\mathbb P(Z>z\mid\delta_i=1)
=
S_Z(z),
\quad z\geqslant 0.
\]
Because the class of half-lines
$
\{(z,\infty):z\geqslant 0\}
$
is a \(\pi\)-system generating \(\mathcal{B}(\mathbb{R}_+)\), the uniqueness
theorem for probability measures implies that
$
\mathbb P(Z\in A\mid\delta_i=1)
=
\mathbb P(Z\in A),
\
A\in\mathcal{B}(\mathbb{R}_+).
$
Therefore,
\[
\mathbb P(Z\in A,\delta_i=1)
=
\mathbb P(Z\in A)\mathbb P(\delta_i=1),
\quad
A\in\mathcal{B}(\mathbb{R}_+).
\]
Since \(\delta_i\) is Bernoulli,
\[
\mathbb P(Z\in A,\delta_i=0)
=
\mathbb P(Z\in A)
-
\mathbb P(Z\in A,\delta_i=1)
=
\mathbb P(Z\in A)\mathbb P(\delta_i=0),
\]
which proves that \(Z\) and \(\delta_i\) are independent.

Conversely, suppose that \(Z\) and \(\delta_i\) are independent. Then,
for every Borel set \(A\),
\[
\mathbb P(Z\in A\mid\delta_i=1)
=
\mathbb P(Z\in A).
\]
Using the identity
$
\mathbb P_{k,i}^*(A)
=
\mathbb P(Z\in A\mid\delta_i=1),
$
proved above, we conclude that
\[
\mathbb P_{k,i}^*(A)
=
\mathbb P(Z\in A),
\quad
A\in\mathcal{B}(\mathbb{R}_+).
\]
Hence the two probability measures coincide, and therefore their survival
functions satisfy
$
S_{k,i}^*(z)
=
S_{k:m}(z),
\ z\geqslant 0.
$
This completes the proof.
\end{proof}

\subsection{Explicit expressions}

In this subsection, we derive explicit expressions for the weighting functions introduced in Theorem \ref{thm:kth-order-general}, together with corresponding representations for the survival function of the $k$th order statistic. Particular attention is devoted to the minimum and maximum order statistics, which play a central role in the right- and left-censoring models studied in the subsequent sections.

\begin{proposition}\label{prop:W-general}
	Let \(X_1,\ldots,X_m\) be independent nonnegative random variables.  Then, for every
	\(i=1,\ldots,m\), $W_{k,i}(x)$ and $S_{k:m}(z)$ given in Theorem \ref{thm:kth-order-general} can be expressed as
	\[
	\begin{aligned}
		W_{k,i}(x)
		=
		\sum_{\substack{
				L,E\subseteq\{1,\ldots,m\}\setminus\{i\}\\
				L\cap E=\varnothing\\
				|L|<k\leqslant |L|+|E|+1
		}}
		\prod_{j\in L}\!\left[1-S_j(x^-)\right]
		\prod_{j\in E}\!\left[S_j(x^-)-S_j(x)\right]
		\prod_{\ell\notin L\cup E\cup\{i\}}S_\ell(x),
			\quad 
		x\geqslant 0,
	\end{aligned}
	\]
	and
		\begin{align*}
		S_{k:m}(z)	
		=
		\sum_{r=m-k+1}^{m}
		\;
		\sum_{\substack{
				B\subset\{1,\ldots,m\}\\
				|B|=r
		}}
		\prod_{i\in B}S_i(z)
		\prod_{j\notin B}\bigl[1-S_j(z)\bigr],
		\quad
		z\geqslant 0,
	\end{align*}	
	where
	$
	S_j(x^-)
	=
	\lim_{y\uparrow x}S_j(y)
	=
	\mathbb P(X_j\geqslant x).
	$
\end{proposition}
\begin{proof}
	Condition on \(X_i=x\) and define the random sets
	$
	L=\{j\neq i:X_j<x\},
	\ 
	E=\{j\neq i:X_j=x\}.
	$
	Then \(L\) and \(E\) are disjoint random subsets of
	\(\{1,\ldots,m\}\setminus\{i\}\). Moreover, since
	\(|L|\) observations are strictly smaller than \(x\) and
	\(|E|+1\) observations (including \(X_i\)) are equal to \(x\),
	the value \(x\) occupies the consecutive positions
	$
	|L|+1,\ldots,|L|+|E|+1
	$
	in the ordered sample. Hence,
	$
	\delta_i=1$
	if and only if
	$
	|L|<k\leqslant |L|+|E|+1.
	$
	
	Therefore,
	\[
	\{\delta_i=1\}
	=
	\bigcup_{\substack{
			A,B\subseteq\{1,\ldots,m\}\setminus\{i\}\\
			A\cap B=\varnothing\\
			|A|<k\leqslant |A|+|B|+1}}
	\{L=A,\;E=B\},
	\]
	where the union is disjoint. Moreover,
	\[
	\{L=A,\;E=B\}
	=
	\bigcap_{j\in A}\{X_j<x\}
	\cap
	\bigcap_{j\in B}\{X_j=x\}
	\cap
	\bigcap_{\ell\notin A\cup B\cup\{i\}}
	\{X_\ell>x\}.
	\]
	The first identity now follows immediately from the independence of
	\(X_1,\ldots,X_m\) and the equalities
	$
	\mathbb P(X_j<x)=1-S_j(x^-),\
	\mathbb P(X_j=x)=S_j(x^-)-S_j(x),\
	\mathbb P(X_j>x)=S_j(x).
	$
	
	The second identity follows immediately from the independence of
	\(X_1,\ldots,X_m\) and the disjoint decomposition
	\[
	\{X_{k:m}>z\}
	=
	\bigcup_{r=m-k+1}^{m}
	\;
	\bigcup_{\substack{
			B\subseteq\{1,\ldots,m\}\\
			|B|=r}}
	\left(
	\bigcap_{i\in B}\{X_i>z\}
	\cap
	\bigcap_{j\notin B}\{X_j\leqslant z\}
	\right).
	\]
	This completes the proof.
\end{proof}
}

\begin{proposition}\label{cor:min-max-general}
	Let \(X_1,\ldots,X_m\) be independent nonnegative random variables. Then, for
	\(i=1,\ldots,m\),
	\[
	W_{1,i}(x)
	=
	\prod_{j\ne i}S_j(x^-),
	\quad
	W_{m,i}(x)
	=
	\prod_{j\ne i}\bigl[1-S_j(x)\bigr],
	\quad
	x\geqslant 0.
	\]
	Moreover,
	\[
	S_{1:m}(z)
	=
	\prod_{i=1}^{m}S_i(z),
	\quad
	S_{m:m}(z)
	=
	1-
	\prod_{i=1}^{m}\bigl[1-S_i(z)\bigr],
	\quad
	z\geqslant 0.
	\]
\end{proposition}
\begin{proof}
	For \(k=1\), the condition
	$
	|L|<1\leqslant |L|+|E|+1
	$
	is equivalent to \(L=\varnothing\). Hence,
	\[
	W_{1,i}(x)
	=
	\sum_{E\subseteq\{1,\ldots,m\}\setminus\{i\}}
	\prod_{j\in E}\!\left[S_j(x^-)-S_j(x)\right]
	\prod_{\ell\notin E\cup\{i\}}S_\ell(x).
	\]
	Applying the distributive law,
	\[
	W_{1,i}(x)
	=
	\prod_{j\ne i}
	\left[
	S_j(x)+S_j(x^-)-S_j(x)
	\right]
	=
	\prod_{j\ne i}S_j(x^-).
	\]
	
	Similarly, for \(k=m\),
	$
	|L|<m\leqslant |L|+|E|+1
	$
	is equivalent to
	$
	L\cup E=\{1,\ldots,m\}\setminus\{i\},
	$
	yielding
	\[
	W_{m,i}(x)
	=
	\sum_{L\subseteq\{1,\ldots,m\}\setminus\{i\}}
	\prod_{j\in L}\!\left[1-S_j(x^-)\right]
	\prod_{\ell\notin L\cup\{i\}}
	\left[S_\ell(x^-)-S_\ell(x)\right].
	\]
	Again, by the distributive law,
	\[
	W_{m,i}(x)
	=
	\prod_{j\ne i}
	\left[
	1-S_j(x^-)+S_j(x^-)-S_j(x)
	\right]
	=
	\prod_{j\ne i}\bigl[1-S_j(x)\bigr].
	\]
	
	The expressions for \(S_{1:m}\) and \(S_{m:m}\) follow immediately from
	Proposition~\ref{prop:W-general}.
\end{proof}

\section{Right-censoring models}\label{sec3}

We now specialize the general framework developed in the previous section to right-censoring mechanisms. In this setting, the observed time corresponds to the minimum order statistic, while the rank indicator identifies the observation responsible for the minimum value. This formulation leads naturally to necessary and sufficient conditions for independence, sufficient conditions based on proportional hazards models, and distribution-free measures of departure from independent censoring.

\subsection{Multiple censoring}

By setting $X_1=Y$, $X_i=C_{i-1}$ for $i=2,\ldots,m$, and taking $i=1$ in Theorem~\ref{thm:kth-order-general}, we obtain a specialization of the general framework to multiple right-censoring models. The following corollary provides a necessary and sufficient condition for the independence between the observed minimum and the associated censoring indicator. An illustrative discrete example is presented at the end of this subsection.

\begin{corollary}\label{cor:multiple-censoring}
	Let \(Y,C_1,\ldots,C_{m-1}\) be independent nonnegative random variables, and define
	\[
	Z=\min\{Y,C_1,\ldots,C_{m-1}\},
	\]
	and
	\[
	\delta
	=
	\mathds1_{\{Y\leqslant C_j,\;j=1,\ldots,m-1\}}.
	\]
	
	Let \(\mathbb P^*\) be the probability measure on \(\mathbb R_+\) defined by
	\[
	\mathbb P^*(A)
	=
	\frac{
		\mathbb E\!\left[
		\prod_{j=1}^{m-1}S_{C_j}(Y^-)
		\mathds1_{\{Y\in A\}}
		\right]
	}{
		\mathbb E\!\left[
		\prod_{j=1}^{m-1}S_{C_j}(Y^-)
		\right]
	},
	\quad
	A\in\mathcal B(\mathbb R_+),
	\]
	and let \(S^*(z)=1-\mathbb P^*((-\infty,z])\) denote its survival function. We have
	\[
	S^*(z)
	=
	S_Y(z)\prod_{j=1}^{m-1}S_{C_j}(z),
	\quad
	z\geqslant 0,
	\]
	if and only if \(Z\) and \(\delta\) are independent.
\end{corollary}
\begin{proof}
    The proof follows by combining Proposition \ref{cor:min-max-general} with Theorem \ref{thm:kth-order-general}.
\end{proof}

%

\begin{example}[Multiple geometric censoring]
Let \(Y,C_1,\ldots,C_{m-1}\) be independent geometric random variables on
\(\{0,1,2,\ldots\}\) with parameters \(p_Y,p_1,\ldots,p_{m-1}\), respectively. Their survival functions are
\[
S_Y(k)=q_Y^{k+1},
\quad
S_{C_j}(k)=q_j^{k+1},
\]
where \(q_Y=1-p_Y\) and \(q_j=1-p_j\).

By Corollary~\ref{cor:multiple-censoring},
\[
\mathbb P^*(\{k\})
\propto
p_Yq_Y^k\prod_{j=1}^{m-1}q_j^k
=
p_Y\theta^k,
\quad
\theta=q_Y\prod_{j=1}^{m-1}q_j,
\]
so that
\[
\mathbb P^*(\{k\})
=(1-\theta)\theta^k,
\quad
k=0,1,\ldots,
\]
that is, \(Y\) remains geometrically distributed under \(\mathbb P^*\). Consequently,
\[
S^*(k)
=
\theta^{k+1}
=
S_Y(k)\prod_{j=1}^{m-1}S_{C_j}(k),
\]
which is precisely the condition of Corollary~\ref{cor:multiple-censoring}. Therefore,
\[
Z=\min\{Y,C_1,\ldots,C_{m-1}\}
\]
is independent of
\[
\delta=\mathds1_{\{Y\leqslant C_j,\;j=1,\ldots,m-1\}},
\]
or equivalently,
\[
\mathbb P(\delta=1\mid Z=k)
=
\mathbb P(\delta=1),
\quad
k=0,1,\ldots.
\]
\end{example}

\subsection{Proportional hazards models}\label{sec5}

The characterization obtained for right-censoring is now applied to
identify broad classes of continuous distributions for which the observed
time and the censoring indicator are automatically independent. This is
achieved under the proportional hazards assumption, leading to a simple
and easily verifiable sufficient condition that encompasses several
classical lifetime models.

\begin{theorem}[Proportional hazards]\label{thm:multiple-PH}
	Let \(Y,C_1,\ldots,C_{m-1}\) be independent absolutely continuous nonnegative random variables, and define
	\[
	Z=\min\{Y,C_1,\ldots,C_{m-1}\},
	\quad
	\delta
	=
	\mathds{1}_{\{Y\leqslant C_j,\;j=1,\ldots,m-1\}}.
	\]
	
	Assume that, for each \(j=1,\ldots,m-1\),
	\[
	h_{C_j}(t)
	=
	\alpha_j h_Y(t),
	\quad
	t\geqslant 0,
	\]
	where \(h\) denotes the hazard rate and \(\alpha_j>0\).
	
	Then \(Z\) and \(\delta\) are independent.
\end{theorem}
\begin{proof}
	Since
	$
	h_{C_j}(t)=\alpha_jh_Y(t),
	$
	we have
	$
	S_{C_j}(t)=S_Y(t)^{\alpha_j},
	\
	j=1,\ldots,m-1.
	$
	Hence
	\[
	\prod_{j=1}^{m-1}S_{C_j}(t)
	=
	S_Y(t)^{\alpha},
	\quad
	\alpha=\sum_{j=1}^{m-1}\alpha_j.
	\]
	Therefore,
	\[
	f^*(t)
	=
	\frac{h_Y(t)S_Y(t)^{1+\alpha}}
	{\mathbb E[S_Y(Y)^\alpha]}.
	\]
	Moreover,
	\[
	\mathbb E[S_Y(Y)^\alpha]
	=
	\int_0^\infty h_Y(t)S_Y(t)^{1+\alpha}{\rm d}t
	=
	\int_0^1u^\alpha{\rm d}u
	=
	\frac1{1+\alpha},
	\]
	where we used the change of variable \(u=S_Y(t)\). Thus,
	\[
	f^*(t)
	=
	(1+\alpha)h_Y(t)S_Y(t)^{1+\alpha}
	=
	-\frac{{\rm d}}{{\rm d}t}S_Y(t)^{1+\alpha},
	\]
	which implies
	\[
	S^*(t)
	=
	S_Y(t)^{1+\alpha}
	=
	S_Y(t)\prod_{j=1}^{m-1}S_{C_j}(t)
	=
	S_Z(t).
	\]
	The result follows from Corollary~\ref{cor:multiple-censoring}.
\end{proof}

Theorem~\ref{thm:multiple-PH} applies to any family of absolutely continuous
distributions that is closed under proportional hazards. Table~\ref{tab:PH_models}
lists several common examples. In each case, if \(Y\) and
\(C_1,\ldots,C_{m-1}\) belong to the same family and share the common
shape (or baseline) parameters indicated in the last column, then
\[
h_{C_j}(t)=\alpha_jh_Y(t),
\quad
\alpha_j>0,
\]
where \(\alpha_j\) is simply the ratio of the corresponding scale (or rate)
parameters.
\begin{table}[H]
	\centering
	\caption{Continuous families satisfying the assumptions of Theorem~\ref{thm:multiple-PH}.}
	\label{tab:PH_models}
	\renewcommand{\arraystretch}{1.25}
		\resizebox{\textwidth}{!}{%
	\begin{tabular}{llll}
		\toprule
		\textbf{Family}
		&
		\textbf{Survival function \(S(t)\)}
		&
		\textbf{Hazard rate \(h(t)\)}
		&
		\textbf{Common parameters}
		\\
		\midrule
		
		Exponential
		&
		\(\exp(-\lambda t)\)
		&
		\(\lambda\)
		&
		None
		\\
		
		Weibull
		&
		\(\exp(-\lambda t^\beta)\)
		&
		\(\lambda\beta t^{\beta-1}\)
		&
		Shape parameter \(\beta\)
		\\
		
		Rayleigh
		&
		\(\exp(-\lambda t^2)\)
		&
		\(2\lambda t\)
		&
		None
		\\
		
		Gompertz
		&
		\(\exp\!\left\{-\dfrac{\lambda}{\gamma}
		\left({\rm e}^{\gamma t}-1\right)\right\}\)
		&
		\(\lambda {\rm e}^{\gamma t}\)
		&
		Shape parameter \(\gamma\)
		\\
		
		Lomax (Pareto II)
		&
		\((1+t/\sigma)^{-\kappa}\)
		&
		\(\dfrac{\kappa}{\sigma+t}\)
		&
		Scale parameter \(\sigma\)
		\\
		
		Burr XII
		&
		\((1+t^c)^{-k}\)
		&
		\(\dfrac{kct^{c-1}}{1+t^c}\)
		&
		Shape parameter \(c\)
		\\
		
		General Cox PH model
		&
		\(\exp\{-\theta H_0(t)\}\)
		&
		\(\theta h_0(t)\)
		&
		Baseline hazard \(h_0\)
		\\
		
		\bottomrule
	\end{tabular}
}
\end{table}
\begin{remark}
	For every family in Table~\ref{tab:PH_models}, let \(Y\) have parameter
	\(\theta_Y\) and \(C_j\) have parameter \(\theta_j\), while sharing the
	common parameters listed in the last column. Then
	\[
	h_{C_j}(t)
	=
	\alpha_j h_Y(t),
	\quad
	\alpha_j=\frac{\theta_j}{\theta_Y},
	\]
	so that the assumptions of Theorem~\ref{thm:multiple-PH} are satisfied.
	For example,
	$
	\alpha_j=
	{\lambda_j}/{\lambda_Y}
	$
	for the exponential, Weibull, Rayleigh and Gompertz families,
	$
	\alpha_j=
	{\kappa_j}/{\kappa_Y}
	$
	for the Lomax family, and
	$
	\alpha_j=
	{k_j}/{k_Y}
	$
	for the Burr XII family.
\end{remark}

The following proposition provides a complete characterization of the independence between the observed time $Z$ and the censoring indicator $\delta$ in the classical right-censoring model. It shows that this independence is equivalent to the proportional hazards assumption, thereby establishing a direct probabilistic characterization of proportional hazards.

\begin{proposition}[Characterization by proportional hazards]
\label{thm:PH-characterization}
Let \(Y\) and \(C\) be independent absolutely continuous nonnegative
random variables, and define
\[
Z=\min\{Y,C\},
\quad
\delta=\mathds{1}_{\{Y\leqslant C\}}.
\]
Then \(Z\) and \(\delta\) are independent if and only if there exists a
constant \(\alpha>0\) such that
\[
h_C(t)=\alpha h_Y(t),
\quad t\geqslant0.
\]
\end{proposition}
\begin{proof}
The sufficiency follows immediately from
Theorem~\ref{thm:multiple-PH}.

Conversely, suppose that \(Z\) and \(\delta\) are independent.
By Corollary~\ref{cor:multiple-censoring},
\[
S^*(t)=S_Y(t)S_C(t),
\quad t\geqslant 0,
\]
where
$
S^*(t)
=
{
\int_t^\infty S_C(x)f_Y(x){\rm d}x}/
{
\int_0^\infty S_C(x)f_Y(x){\rm d}x}.
$
Differentiating both sides yields
\[
\frac{S_C(t)f_Y(t)}
{\mathbb E[S_C(Y)]}
=
S_C(t)f_Y(t)+S_Y(t)f_C(t).
\]
Dividing by \(S_Y(t)S_C(t)\) gives 
$
h_C(t)=\alpha h_Y(t),
$
where
$
\alpha
=
\{1/{\mathbb E[S_C(Y)]}\}-1
>0
$
is constant. This completes the proof.
\end{proof}

\subsection{Measures of departure from independent right-censoring}\label{perfMeasures}

Corollary~\ref{cor:multiple-censoring} naturally gives rise to distribution-free measures of departure from independent right-censoring. These are defined, respectively, as the Kolmogorov and normalized Wasserstein distances between the conditional distribution induced by $\mathbb{P}^*$ and the unconditional distribution of the observed minimum. Specifically,
\[
K_d
=
\sup_{z\geqslant 0}
\left|
\frac{
\displaystyle
\int_{z}^\infty
\prod_{j=1}^{m-1}S_{C_j}(y^-)\,
{\rm d}F_Y(y)
}
{
\displaystyle
\int_0^\infty
\prod_{j=1}^{m-1}S_{C_j}(y^-)\,
{\rm d}F_Y(y)
}
-
S_Y(z)\prod_{j=1}^{m-1}S_{C_j}(z)
\right|
\]
and
\[
W_d^*
=
\frac{
\displaystyle
\int_0^\infty
\left|
\frac{
\displaystyle
\int_{z}^\infty
\prod_{j=1}^{m-1}S_{C_j}(y^-)\,
{\rm d}F_Y(y)
}
{
\displaystyle
\int_0^\infty
\prod_{j=1}^{m-1}S_{C_j}(y^-)\,
{\rm d}F_Y(y)
}
-
S_Y(z)\prod_{j=1}^{m-1}S_{C_j}(z)
\right|
{\rm d}z
}
{
\displaystyle
\max\left\{
\frac{
\displaystyle
\int_0^\infty
y\prod_{j=1}^{m-1}S_{C_j}(y^-)\,
{\rm d}F_Y(y)
}
{
\displaystyle
\int_0^\infty
\prod_{j=1}^{m-1}S_{C_j}(y^-)\,
{\rm d}F_Y(y)
},
\;
\int_0^\infty
S_Y(z)\prod_{j=1}^{m-1}S_{C_j}(z)\,
{\rm d}z
\right\}
}.
\]

Both measures satisfy
$
0\leqslant K_d, W_d^*\leqslant 1,
$
and
$
K_d=0
\Longrightarrow
W_d^*=0,
$
which, by Corollary~\ref{cor:multiple-censoring}, is equivalent to independent right-censoring. Thus, $K_d$ measures the largest discrepancy between the conditional survival function under $\mathbb{P}^*$ and the survival function of the observed minimum, whereas $W_d^*$ quantifies the overall departure through the corresponding normalized Wasserstein distance of order one.

\section{Left-censoring models}\label{sec4}

We next consider the dual problem associated with left-censoring mechanisms. In contrast to the previous section, the observed time is defined through the maximum order statistic, and the rank indicator identifies the observation attaining the largest value. Building on the general results established in Section \ref{sec2}, we derive characterizations of independence, establish sufficient conditions based on proportional reversed hazards models, and introduce corresponding measures of departure from independence.

\subsection{Multiple censoring}

By considering the maximum order statistic and setting $X_1=Y$, $X_i=C_{i-1}$ for $i=2,\ldots,m$, and $i=1$ in Theorem~\ref{thm:kth-order-general}, we obtain the corresponding characterization for multiple left-censoring models. The following corollary provides a necessary and sufficient condition for the independence between the observed maximum and the associated censoring indicator. An illustrative example is presented at the end of this subsection.

\begin{corollary}\label{cor:multiple-left-censoring}
	Let \(Y,C_1,\ldots,C_{m-1}\) be independent nonnegative random variables, and define
	\[
	Z=\max\{Y,C_1,\ldots,C_{m-1}\},
	\]
	and
	\[
	\delta
	=
	\mathds1_{\{Y\geqslant C_j,\;j=1,\ldots,m-1\}}.
	\]
	
	Let \(\mathbb P^*\) be the probability measure on \(\mathbb R_+\) defined by
	\[
	\mathbb P^*(A)
	=
	\frac{
		\mathbb E\!\left[
		\prod_{j=1}^{m-1}\bigl(1-S_{C_j}(Y)\bigr)
		\mathds1_{\{Y\in A\}}
		\right]
	}{
		\mathbb E\!\left[
		\prod_{j=1}^{m-1}\bigl(1-S_{C_j}(Y)\bigr)
		\right]
	},
	\quad
	A\in\mathcal B(\mathbb R_+),
	\]
	and let \(S^*(z)=1-\mathbb P^*((-\infty,z])\) denote its survival function.
	
	We have
	\[
	S^*(z)
	=
	1-
	\bigl[1-S_Y(z)\bigr]
	\prod_{j=1}^{m-1}
	\bigl[1-S_{C_j}(z)\bigr],
	\quad
	z\geqslant 0,
	\]
	if and only if \(Z\) and \(\delta\) are independent.
\end{corollary}


\begin{example}[Multiple left-censoring under a discrete proportional reversed hazards family]
Let \(F_0\) be the distribution function of a nonnegative integer-valued random variable, and define
\[
F_\theta(k)=F_0(k)^\theta,
\quad
k\in\mathbb N_0,\ \theta>0.
\]
Suppose that
$
Y\sim F_{\theta_Y}
$
and
$
C_j\sim F_{\theta_j},
$
\(j=1,\ldots,m-1\), are mutually independent. Then
\[
F_Y(k)\prod_{j=1}^{m-1}F_{C_j}(k)
=
F_0(k)^{\theta_Y+\sum_{j=1}^{m-1}\theta_j}.
\]
Moreover, under the probability measure \(\mathbb P^*\) of
Corollary~\ref{cor:multiple-left-censoring},
\[
F^*(k)
=
F_0(k)^{\theta_Y+\sum_{j=1}^{m-1}\theta_j},
\]
so that
\[
S^*(k)
=
1-
F_Y(k)\prod_{j=1}^{m-1}F_{C_j}(k)
=
1-
\bigl[1-S_Y(k)\bigr]
\prod_{j=1}^{m-1}
\bigl[1-S_{C_j}(k)\bigr].
\]
Hence, the assumptions of Corollary~\ref{cor:multiple-left-censoring} are satisfied, implying that
\[
Z=\max\{Y,C_1,\ldots,C_{m-1}\}
\]
is independent of
\[
\delta=\mathds1_{\{Y\geqslant C_j,\;j=1,\ldots,m-1\}},
\]
or equivalently,
\[
\mathbb P(\delta=1\mid Z=k)
=
\mathbb P(\delta=1),
\quad
k\in\mathbb N_0.
\]
\end{example}

\subsection{Proportional reversed hazards models}\label{sec6}

We next consider the dual setting based on proportional reversed hazards
models. By combining the characterization for the maximum order statistic
with the reversed hazards assumption, we obtain sufficient conditions
ensuring independence in left-censoring models. The results parallel
those established in the previous section and cover several familiar
distribution families.

\begin{theorem}[Proportional reversed hazards]
	\label{thm:multiple-RPH}
	Let \(Y,C_1,\ldots,C_{m-1}\) be independent absolutely continuous nonnegative random variables, and define
	\[
	Z=\max\{Y,C_1,\ldots,C_{m-1}\},
	\quad
	\delta
	=
	\mathds{1}_{\{Y\geqslant C_j,\;j=1,\ldots,m-1\}}.
	\]
	
	Assume that, for each \(j=1,\ldots,m-1\),
	\[
	r_{C_j}(t)
	=
	\beta_j r_Y(t),
	\quad
	t\geqslant 0,
	\]
	where \(r\) denotes the reversed hazard rate and \(\beta_j>0\).
	
	Then \(Z\) and \(\delta\) are independent.
\end{theorem}
\begin{proof}
	Since
	$
	r_{C_j}(t)=\beta_jr_Y(t),
	$
	we have
	$
	F_{C_j}(t)=F_Y(t)^{\beta_j},
	\
	j=1,\ldots,m-1.
	$
	Hence
	\[
	\prod_{j=1}^{m-1}F_{C_j}(t)
	=
	F_Y(t)^{\beta},
	\quad
	\beta=\sum_{j=1}^{m-1}\beta_j.
	\]
	Therefore,
	\[
	f^*(t)
	=
	\frac{r_Y(t)F_Y(t)^{1+\beta}}
	{\mathbb E[F_Y(Y)^\beta]}.
	\]
	Moreover,
	\[
	\mathbb E[F_Y(Y)^\beta]
	=
	\int_0^\infty r_Y(t)F_Y(t)^{1+\beta}{\rm d}t
	=
	\int_0^1u^\beta{\rm d}u
	=
	\frac1{1+\beta},
	\]
	where we used the change of variable \(u=F_Y(t)\). Thus,
	\[
	f^*(t)
	=
	(1+\beta)r_Y(t)F_Y(t)^{1+\beta}
	=
	\frac{{\rm d}}{{\rm d}t}F_Y(t)^{1+\beta},
	\]
	which implies
	\[
	F^*(t)
	=
	F_Y(t)^{1+\beta}.
	\]
	Equivalently,
	\[
	S^*(t)
	=
	1-F_Y(t)^{1+\beta}
	=
	1-
	F_Y(t)\prod_{j=1}^{m-1}F_{C_j}(t)
	=
	S_Z(t).
	\]
	The result follows from Corollary~\ref{cor:multiple-left-censoring}.
\end{proof}

Theorem~\ref{thm:multiple-RPH} applies to any family of absolutely
continuous distributions that is closed under proportional reversed
hazards. Table~\ref{tab:PRH_models} lists several common examples.
In each case, if \(Y\) and \(C_1,\ldots,C_{m-1}\) belong to the same
family and share the common shape (or baseline) parameters indicated
in the last column, then
\[
r_{C_j}(t)=\beta_jr_Y(t),
\quad
\beta_j>0,
\]
where \(\beta_j\) is the ratio of the corresponding exponent (or shape)
parameters.
\begin{table}[H]
	\centering
	\small
	\caption{Continuous families satisfying the assumptions of Theorem~\ref{thm:multiple-RPH}.}
	\label{tab:PRH_models}
	\renewcommand{\arraystretch}{2.3}
	\resizebox{\textwidth}{!}{%
		\begin{tabular}{llll}
			\toprule
			\textbf{Family}
			&
			\textbf{Distribution function \(F(t)\)}
			&
			\textbf{Reversed hazard rate \(r(t)\)}
			&
			\textbf{Common parameters}
			\\
			\midrule
			
			Power
			&
			\(t^\theta,\quad 0<t<1\)
			&
			\(\dfrac{\theta}{t}\)
			&
			Support \((0,1)\)
			\\
			
			Beta\((\theta,1)\)
			&
			\(t^\theta,\quad 0<t<1\)
			&
			\(\dfrac{\theta}{t}\)
			&
			Second shape parameter fixed at \(1\)
			\\
			
			Pareto I
			&
			\(\left(1-\dfrac{\sigma}{t}\right)^\theta,\quad t>\sigma\)
			&
			\(\dfrac{\theta\sigma}{t(t-\sigma)}\)
			&
			Scale parameter \(\sigma\)
			\\
			
			Exponentiated Exponential
			&
			\(\left(1-{\rm e}^{-\lambda t}\right)^\theta\)
			&
			\(\dfrac{\theta\lambda {\rm e}^{-\lambda t}}
			{1-{\rm e}^{-\lambda t}}\)
			&
			Rate parameter \(\lambda\)
			\\
			
			Exponentiated Weibull
			&
			\(\left(1-{\rm e}^{-\lambda t^\beta}\right)^\theta\)
			&
			\(\dfrac{\theta\lambda\beta t^{\beta-1}
				{\rm e}^{-\lambda t^\beta}}
			{1-{\rm e}^{-\lambda t^\beta}}\)
			&
			Parameters \(\lambda,\beta\)
			\\
			
			Exponentiated Rayleigh
			&
			\(\left(1-{\rm e}^{-\lambda t^2}\right)^\theta\)
			&
			\(\dfrac{2\theta\lambda t
				{\rm e}^{-\lambda t^2}}
			{1-{\rm e}^{-\lambda t^2}}\)
			&
			Parameter \(\lambda\)
			\\
			
			General PRH model
			&
			\(F_0(t)^\theta\)
			&
			\(\theta r_0(t)\)
			&
			Baseline reversed hazard \(r_0\)
			\\
			
			\bottomrule
		\end{tabular}
	}
\end{table}
\begin{remark}
	For every family in Table~\ref{tab:PRH_models}, let \(Y\) have parameter
	\(\theta_Y\) and let \(C_j\) have parameter \(\theta_j\), while sharing the
	common parameters listed in the last column. Then
	\[
	r_{C_j}(t)
	=
	\beta_jr_Y(t),
	\quad
	\beta_j=\frac{\theta_j}{\theta_Y},
	\]
	so that the assumptions of Theorem~\ref{thm:multiple-RPH} are satisfied.
	
	In particular,
	\[
	F_{C_j}(t)
	=
	F_Y(t)^{\beta_j},
	\]
	which follows from integrating the proportional reversed hazards model.
	Consequently,
	\[
	F_Z(t)
	=
	F_Y(t)
	\prod_{j=1}^{m-1}F_{C_j}(t)
	=
	F_Y(t)^{\,1+\sum_{j=1}^{m-1}\beta_j},
	\]
	where
	\[
	Z=\max\{Y,C_1,\ldots,C_{m-1}\},
	\]
	showing that the hypothesis of Theorem~\ref{thm:multiple-RPH} holds for
	all the families listed in Table~\ref{tab:PRH_models}.
\end{remark}

The following result is the dual counterpart of Proposition~\ref{thm:PH-characterization}. It shows that, for $m=2$, the independence between $Z$ and $\delta$ is equivalent to the proportional reversed hazards assumption.

\begin{proposition}[Characterization by proportional reversed hazards]
\label{thm:RPH-characterization}
Let \(Y\) and \(C\) be independent absolutely continuous nonnegative
random variables, and define
\[
Z=\max\{Y,C\},
\quad
\delta=\mathds{1}_{\{Y\geqslant C\}}.
\]
Then \(Z\) and \(\delta\) are independent if and only if there exists a
constant \(\beta>0\) such that
\[
r_C(t)=\beta r_Y(t),
\quad t\geqslant0,
\]
where \(r_Y\) and \(r_C\) denote the reversed hazard rate functions of
\(Y\) and \(C\), respectively.
\end{proposition}
\begin{proof}
The sufficiency follows immediately from Theorem~\ref{thm:multiple-RPH}
with \(m=2\).

Conversely, suppose that \(Z\) and \(\delta\) are independent. Then
\[
\mathbb{P}(Z\leqslant z,\delta=1)
=
F_Z(z)\mathbb{P}(\delta=1).
\]
Since
$
\mathbb{P}(Z\leqslant z,\delta=1)
=
\int_0^z f_Y(t)F_C(t){\rm d}t
$
and
$
P(Z\leqslant z)
=
F_Y(z)F_C(z),
$
we obtain
\[
\int_0^z f_Y(t)F_C(t){\rm d}t
=
\mathbb{P}(\delta=1)F_Y(z)F_C(z).
\]
Differentiating both sides with respect to \(z\) yields
\[
f_Y(z)F_C(z)
=
\mathbb{P}(\delta=1)
\bigl[f_Y(z)F_C(z)+f_C(z)F_Y(z)\bigr].
\]
Dividing both sides by \(F_Y(z)F_C(z)\) gives
$r_C(z)=\beta r_Y(z),$
where $\beta=[{1-\mathbb{P}(\delta=1)}]/{\mathbb{P}(\delta=1)}$.
This completes the proof.
\end{proof}

\subsection{Measures of departure from independent left-censoring}

Corollary~\ref{cor:multiple-left-censoring} also leads naturally to distribution-free measures that quantify the degree of departure from independent left-censoring. These measures are obtained by comparing the conditional distribution under $\mathbb{P}^*$ with the distribution of the observed maximum through the Kolmogorov and normalized Wasserstein distances. Specifically,
\[
K_d
=
\sup_{z\geqslant 0}
\left|
\frac{
\displaystyle
\int_z^\infty
\prod_{j=1}^{m-1}\bigl[1-S_{C_j}(y)\bigr]\,
{\rm d}F_Y(y)
}
{
\displaystyle
\int_0^\infty
\prod_{j=1}^{m-1}\bigl[1-S_{C_j}(y)\bigr]\,
{\rm d}F_Y(y)
}
-
\left[
1-
\bigl[1-S_Y(z)\bigr]
\prod_{j=1}^{m-1}
\bigl[1-S_{C_j}(z)\bigr]
\right]
\right|
\]
and
\[
W_d^*
=
\frac{
\displaystyle
\int_0^\infty
\left|
\frac{
\displaystyle
\int_z^\infty
\prod_{j=1}^{m-1}\bigl[1-S_{C_j}(y)\bigr]\,
{\rm d}F_Y(y)
}
{
\displaystyle
\int_0^\infty
\prod_{j=1}^{m-1}\bigl[1-S_{C_j}(y)\bigr]\,
{\rm d}F_Y(y)
}
-
\left[
1-
\bigl[1-S_Y(z)\bigr]
\prod_{j=1}^{m-1}
\bigl[1-S_{C_j}(z)\bigr]
\right]
\right|
{\rm d}z
}
{
\displaystyle
\max\left\{
\frac{
\displaystyle
\int_0^\infty
y\prod_{j=1}^{m-1}\bigl[1-S_{C_j}(y)\bigr]\,
{\rm d}F_Y(y)
}
{
\displaystyle
\int_0^\infty
\prod_{j=1}^{m-1}\bigl[1-S_{C_j}(y)\bigr]\,
{\rm d}F_Y(y)
},
\;
\int_0^\infty
\left[
1-
\bigl[1-S_Y(z)\bigr]
\prod_{j=1}^{m-1}
\bigl[1-S_{C_j}(z)\bigr]
\right]
{\rm d}z
\right\}
}.
\]

Both measures take values in the unit interval,
$
0\leqslant K_d, W_d^*\leqslant 1,
$
and satisfy
$
K_d=0
\Longrightarrow
W_d^*=0,
$
which, by Corollary~\ref{cor:multiple-left-censoring}, characterizes independent left-censoring. Moreover, $K_d$ captures the largest discrepancy between the conditional distribution under $\mathbb{P}^*$ and the distribution of the observed maximum, whereas $W_d^*$ provides a global measure of this discrepancy through the normalized Wasserstein distance of order one.

\section{Numerical Illustration}\label{sec7}

This section presents a numerical investigation of the theoretical results established throughout the paper. Using Monte Carlo simulations, we evaluate the proposed characterizations under several right- and left-censoring scenarios and examine the performance of the associated distance measures. The results provide additional insight into the conditions under which independence holds and highlight the fundamental differences between the probabilistic structures induced by the minimum and maximum operators.

To illustrate the theoretical findings, this section presents a Monte Carlo (MC) simulation study assessing the proposed characterizations under different sampling scenarios. {To this end, the independence between $Z$ and $\delta$ can be characterized using the result established in Subsection~\ref{perfMeasures}. 

Without loss of generality, consider the case where $m=5$, in which $Y_i$ denotes the failure time and $C_{ij}$, for $j\in \{1,2,3,4\}$ and $i \in \{1,2,\cdots,n\}$, represent the censoring times. In addition, the failure and censoring times were generated from Weibull and Exponential distributions, such that $Y_i \sim \mathrm{Weibull}(\alpha, \phi)$, where $\alpha>0$ and $\phi>0$ are the shape and scale parameters, respectively. Let $C_{ij} \sim \exp(\mu_j)$, where $\mu_j > 0$ is the rate parameter. The observed survival times and failure indicators are defined as follows: 

\vspace{-0.2cm}
\begin{itemize}
\item For the right-censoring mechanism: $Z_i = \min\left\{Y_i, C_i\right\}$ and $\delta_i = \mathds{1}_{\{Y_i \leqslant C_i\}}$, where $C_i = \min\left\{C_{i1}, C_{i2}, C_{i3}, C_{i4}\right\}$ is the overall censoring time.\vspace{-0.1cm}

\item For the left-censoring mechanism: $Z_i = \max\left\{Y_i, C_i\right\}$ and $\delta_i = \mathds{1}_{\{Y_i \geqslant C_i\}}$, where $C_i = \max\left\{C_{i1}, C_{i2}, C_{i3}, C_{i4}\right\}$ is the overall censoring time.
\end{itemize}

The means of these random variables are given by $\mathbb{E}[Y_i]=\phi^{-1}\Gamma(1+1/\alpha)$ and $\mathbb{E}[C_{ij}] = 1/\mu_j$, respectively.

The independence between $Z_i$ and $\delta_i$ was assessed by evaluating the proximity between the conditional survival curves $S_{Z_i|\delta_i = 1}(z_i)$ and $S_{Z_i|\delta_i = 0}(z_i)$ across $R=100$ MC replications. To examine the effectiveness of the proposed methodology, six scenarios were considered, each characterized by different parameter values: $\alpha$, $\phi$, and $\mu_G = \sum_{j=1}^{4} \mu_j$. For each scenario, $R$ datasets were generated, each with a fixed sample size of $n=1000$. In addition, the following proximity measures were computed: $(i)$ $\mathbb{E}[Z_i|\delta_i=\kappa]$, the conditional expectation, for $\kappa \in \{0,1\}$ and $\forall i \in \{1,2,\ldots,n\}$;  $(ii)$ $W_d$, the Wasserstein distance; $(iii)$ $W_d^*$, the normalized Wasserstein distance; and $(iv)$ $K_d$, the Kolmogorov distance. Additionally, the probability that an individual experiences the event of interest, $\mathbb{P}(\delta_i = 1)$, was also reported. For these similarity measures, values close to zero indicate independence between the two variables. In other words, under this assumption, discrepancies between the conditional distributions (or survival functions) should be negligible, with the corresponding distance measures expected to be close to zero.


The simulation results presented in Table~\ref{TabsimuRes} are highly consistent with the theoretical properties of independence between the variables $Z_i$ and $\delta_i$. However, under the specified distributions, this independence is governed primarily by the type of censoring mechanism, the distribution of the failure time $Y_i$, and, in particular, the value of the Weibull shape parameter $\alpha$.

\begin{table}[H]
\centering
\setlength{\tabcolsep}{0.15cm}
\small
\caption{Monte Carlo simulation results for the main performance measures. Here, $W_d$, $W_d^*$, and $K_d$ denote the Wasserstein distance, the normalized Wasserstein distance, and the Kolmogorov distance, respectively. $\mathbb{P}(\delta_i = 1)$ represents the failure probability, and $\mathbb{E}\left[Z_i \mid \delta_i = \kappa\right]$ for $\kappa=0,1$ and $i = 1, 2, \dots, n$ denotes the conditional expectation.}
\label{TabsimuRes}
\vspace{0.15cm}
\begin{tabular}{ccccccccccc}
\hline 
 \multirow{2}{*}{Scenario} & \multicolumn{9}{c}{\bf Right-Censoring Case}\\
 \cline{2-10}
  & $\alpha$ & $\phi$ & $\mu_G$ & $\mathbb{E}\left[Z_i|\delta_i=1\right]$ & $\mathbb{E}\left[Z_i|\delta_i=0\right]$& $\mathbb{P}\left(\delta_i=1\right)$& $W_d$ & $W_d^*$ & $K_d$\\ 
  \hline 
  1 &1.0 &0.10 &0.10 &5.000 &5.000  &0.500 &0.000 &0.000 &0.000\\
  2 &1.5 &0.10 &0.10 &6.236 &4.407  &0.473 &0.341 &0.055 &0.196\\
  3 &0.5 &0.10 &0.10 &2.918 &6.495  &0.546 &2.682 &0.413 &0.325\\
  4 &1.0 &0.05 &0.01 &16.667&16.667 &0.833 &0.000 &0.000 &0.000\\
  5 &0.5 &0.05 &0.05 &5.836 &12.991 &0.546 &2.586 &0.199 &0.325\\
  6 &1.5 &0.05 &0.20 &6.005 & 4.092 &0.131 &0.394 &0.066 &0.205\\
\hline
\multirow{2}{*}{Scenario} & \multicolumn{9}{c}{\bf Left-Censoring Case}\\
  \cline{2-10}
  & $\alpha$ & $\phi$ & $\mu_G$ & $\mathbb{E}\left[Z_i|\delta_i=1\right]$ & $\mathbb{E}\left[Z_i|\delta_i=0\right]$& $\mathbb{P}\left(\delta_i=1\right)$& $W_d$ & $W_d^*$ & $K_d$\\ 
\hline
 1& 1.0& 0.10& 0.10& 15.000& 15.000 & 0.500& 0.000 & 0.000 & 0.000\\
 2& 1.5& 0.10& 0.10& 11.531& 16.236 & 0.527& 30.594& 1.884 & 0.192\\
 3& 0.5& 0.10& 0.10& 40.514& 12.918 & 0.454& 30.934 & 0.764& 0.320\\
 4& 1.0& 0.05& 0.01& 36.667& 116.667& 0.167& 346.532& 2.970& 0.464\\
 5& 0.5& 0.05& 0.05& 81.027& 25.836 & 0.454& 57.915 & 0.715& 0.320\\
 6& 1.5& 0.05& 0.20& 19.877& 11.005 & 0.869& 32.643 & 1.642& 0.360\\
  \hline 
\end{tabular}
\end{table}

For the right-censoring model characterization, the results from Scenarios 1 and 4 ($\alpha = 1$) where the Weibull distribution reduces to the exponential distribution and the censoring times follow a similar distribution, show that the conditional distributions of $Z_i$ given $\delta_i = 1$ and $\delta_i = 0$ coincide. This is demonstrated by the equality of the conditional expectations, that is, $\mathbb{E}\left[Z_i|\delta_i=1\right]=\mathbb{E}[Z_i|\delta_i=0],\,\, \forall i \in \{1,2,\ldots,n\}$, as well as by the Wasserstein and Kolmogorov distances both being zero. This behavior, which can be interpreted as a direct consequence of the memoryless property of the exponential distribution, is further illustrated by the conditional survival curves $S_{Z_i|\delta_i = 1}(z_i)$ and $S_{Z_i|\delta_i = 0}(z_i)$, which virtually coincide, as shown in Figure~\ref{figMin}.

The results further show that in Scenarios 2 and 6, where $\alpha = 1.5$, the Weibull distribution exhibits an increasing hazard function. In this case, the conditional expectations satisfy $\mathbb{E}\left[Z_i|\delta_i=1\right] > \mathbb{E}\left[Z_i|\delta_i=0\right]$, indicating that individuals who experience the event tend to have larger observed survival times, on average, than those who are censored. Consequently, the conditional survival functions diverge moderately, with the performance measures leading to values near zero. The maximum discrepancy between the conditional survival functions increases when the hazard rate exhibits a decreasing hazard rate, $\alpha = 0.5$ (Scenarios 3 and 5). In this case, the opposite pattern emerges: $\mathbb{E}\left[Z_i|\delta_i=1\right] < \mathbb{E}\left[Z_i|\delta_i=0\right]$, indicating that individuals who experience the event tend, on average, to have small observed times than those who are censored. These findings demonstrate that $Z_i$ and $\delta_i$ are no longer independent, especially when the hazard rate decreases over time, as evidenced by the results shown in Table~\ref{TabsimuRes} and panels 2, 3, 5, and 6 of Figure~\ref{figMin}.

\begin{figure}[H]
\centering
\includegraphics[width=1.0\linewidth]{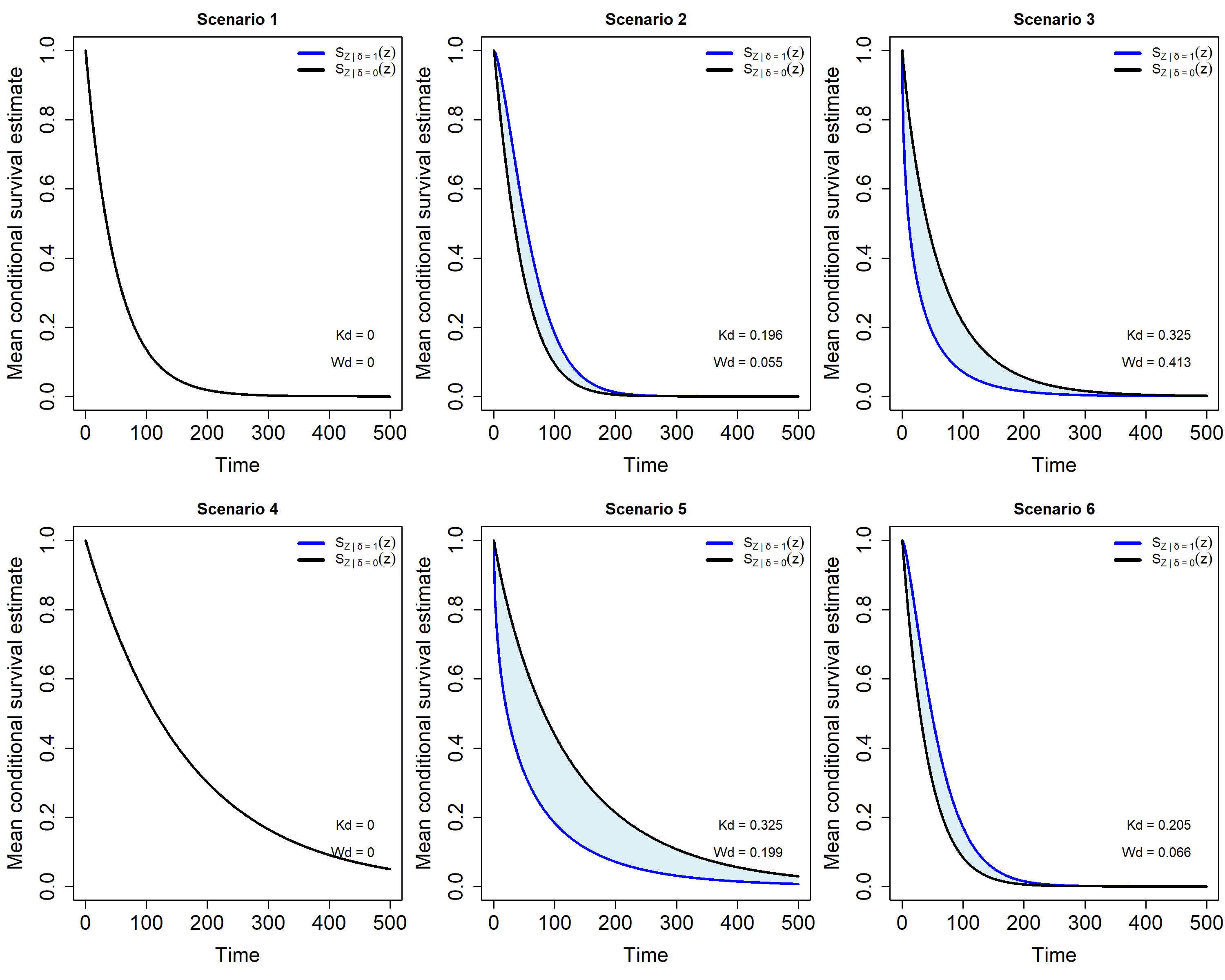}\\
\vspace{-0.3cm}
\caption{Conditional survival functions under the right-censoring characterization are shown for six simulation scenarios. The blue and black curves represent $S_{Z|\delta=1}(z)$ and $S_{Z|\delta=0}(z)$, respectively. The shaded region highlights the discrepancy between these conditional survival curves.}
\label{figMin}
\end{figure}

\begin{figure}[H]
\centering
\includegraphics[width=1.0\linewidth]{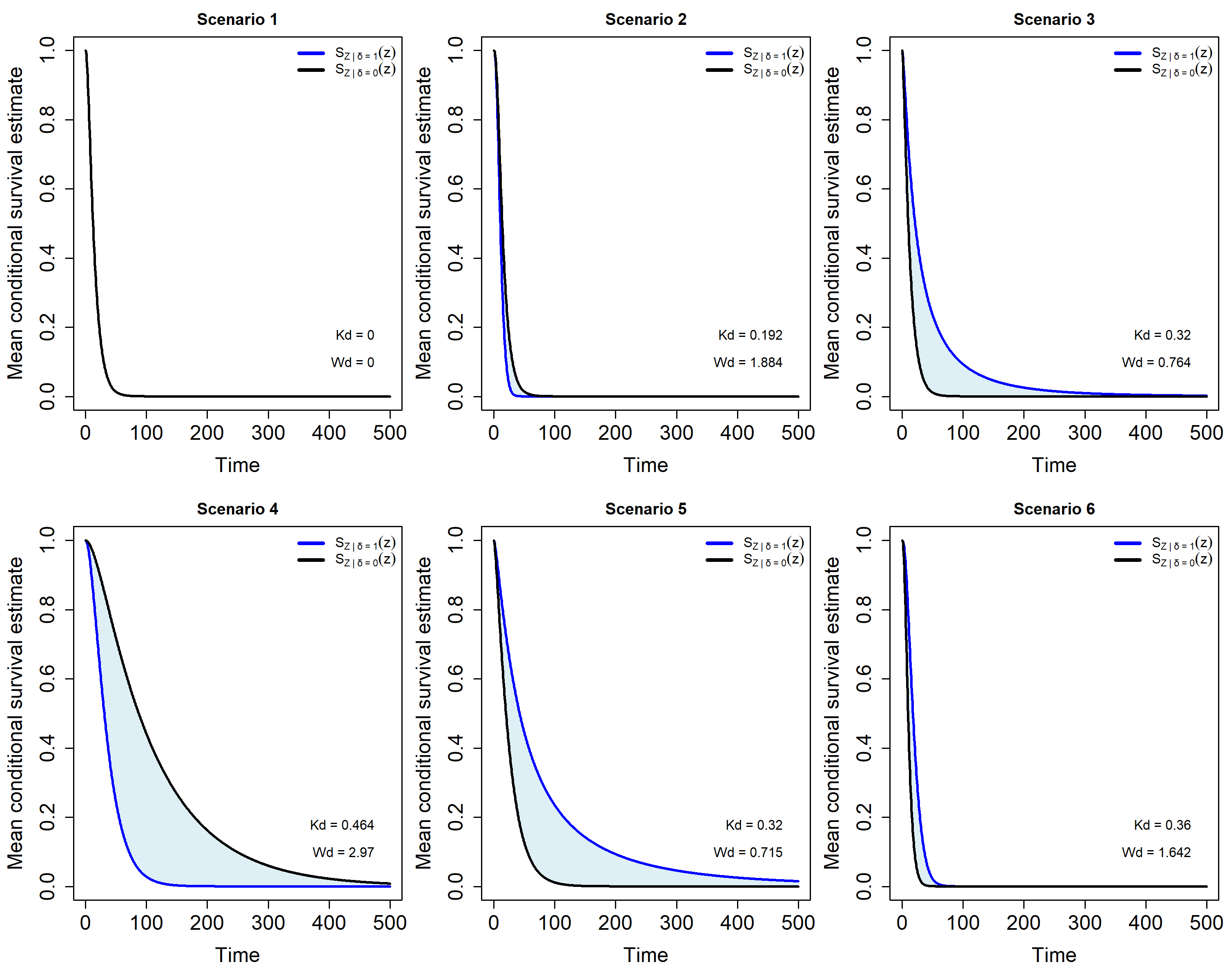}\\
\vspace{-0.3cm}
\caption{Conditional survival functions under the left-censoring characterization are shown for six simulation scenarios. The blue and black curves represent $S_{Z|d=1}(z)$ and $S_{Z|d=0}(z)$, respectively. The shaded region highlights the discrepancy between these conditional survival curves.}
\label{figMax}
\end{figure}

For the maximum characterization model, the criteria for independence between $Z_i$ and $\delta_i$ are particularly stringent. Although setting the Weibull shape parameter to $(\alpha=1)$ is necessary, it is not sufficient. As shown in Table~\ref{TabsimuRes}, independence is achieved only when the additional condition $\phi=\mu_G$ (Scenario 1) is also satisfied. For example, although Scenario 4 fulfills the requirement of $(\alpha=1)$, the fact that $\phi\neq \mu_G$ leads to unequal conditional means. In particular, $\mathbb{E}\left[Z_i|\delta_i=1\right] < \mathbb{E}\left[Z_i|\delta_i=0\right]$, confirming that $S_{Z_i|\delta_i = 1}(z_i)\neq S_{Z_i|\delta_i = 0}(z_i)$ in average. As a result, all distance measures are strictly greater than zero, with the Kolmogorov distance being the smallest among those considered.

Unlike in the right-censoring case, the values of $W_d^{*}$ were lower when the failure times $Y_i$ were generated from a Weibull distribution with a decreasing hazard rate, $0<\alpha<1$, than with an increasing hazard rate, $\alpha>1$. In contrast, the Kolmogorov distance displayed the opposite pattern. The pairs of scenarios (3, 5) yield identical Kolmogorov distance values, even though their conditional means differ substantially. These results indicate that variations in the scale parameter $\phi$ and the global censoring intensity $\mu_G$ have only a negligible impact on the values of $K_d$, further underscoring the robustness of this distance measure with respect to these parameters.

Similarly, all conclusions drawn for the left-censoring case are consistent with the patterns observed in Figure~\ref{figMax}.
}

\subsection{Discussion of numerical results}\label{subsec:discussion}

{
The results show that the Weibull shape parameter, $\alpha$, is the primary factor governing the dependence between survival time $Z_i$ and event indicator $\delta_i$. Specifically, for the distributions considered here, no combination of the scale parameters $\phi$ and $\mu_G$ can ensure independence when $\alpha \neq 1$. This conclusion follows directly from the distinct hazard structures of the failure and censoring mechanisms. Under the classical right-censoring framework, where the observed time is defined by the minimum operator, independence between $Z_i$ and $\delta_i$ is relatively easy to achieve. The necessary and sufficient condition is that both the failure times $Y_i$ and the censoring times $C_j$,  $j=1,\ldots,m-1$, are exponentially distributed. Notably, this condition imposes no restrictions on the associated rate parameters, indicating that independence holds across an entire family of exponential models, regardless of the specific values of $\phi$ and $\mu_G$. Therefore, independence is a robust and naturally occurring property in conventional survival analysis.

A major finding of this study is the contrast between the minimum and maximum operators. Unlike the minimum operator, the maximum operator in the left-censoring framework requires more than exponential failure and censoring distributions to achieve independence. Although exponentiality is necessary, it is not sufficient; independence occurs only when the failure and censoring times are independent and identically distributed, which, under the adopted parametrization, corresponds to $\phi = \mu_G$.

These results demonstrate that achieving independence between the observed time and the event indicator is substantially easier with the minimum operator than with the maximum operator. While the classical survival analysis setting permits independence across a broad range of exponential parameterizations, the maximum construction exhibits this property only under highly restrictive conditions. This fundamental distinction provides a new probabilistic characterization of the two observation mechanisms and shows that the proposed Kolmogorov and Wasserstein distance measures effectively capture their distinct dependence structures.

}


\section{Concluding remarks} \label{FinalR}

This paper developed a general framework for characterizing the independence between an order statistic and its corresponding rank indicator. The proposed approach is based on a probability measure obtained by reweighting the distribution of a single observation according to its conditional probability of occupying a prescribed rank. This construction yields explicit expressions for the associated weighting functions and provides a unified characterization of independence in terms of the equality between conditional and unconditional distributions.

Specializing the general framework to the minimum and maximum order statistics led to new characterizations of independent right- and left-censoring mechanisms, including both single and multiple censoring schemes. In addition, we established broad sufficient conditions based on proportional hazards and proportional reversed hazards models, together with complete characterizations in the classical single-censoring setting. The proposed Kolmogorov and Wasserstein distance measures further provide natural and distribution-free tools for quantifying departures from independence.

The numerical results corroborate the theoretical findings and demonstrate that the dependence structure is fundamentally determined by the interaction between the failure and censoring mechanisms. Although both the minimum and maximum operators are constructed from the same underlying random variables, they exhibit markedly different probabilistic behaviors. In particular, independence is considerably easier to achieve in the right-censoring setting than in the left-censoring setting. Whereas the minimum operator requires only proportional hazards structures---and, in the classical case, exponential failure and censoring distributions---the maximum operator is governed by considerably more restrictive conditions.

These findings reveal a fundamental distinction between the probabilistic structures associated with the minimum and maximum operators, which are governed, respectively, by the hazard and reversed hazard functions. More broadly, the results show that independence is not merely a consequence of the marginal distributions of the failure and censoring times, but rather a structural property determined by the way in which these variables interact to generate the observed data. In this sense, the proposed framework provides both a theoretical foundation and practical tools for investigating dependence structures in survival analysis.

Finally, several directions for future research arise naturally from the present work. An immediate extension would be to investigate analogous characterizations under more general dependence structures, including models with dependent failure and censoring times, frailty models, and copula-based formulations. Another promising avenue would be to extend the proposed methodology to progressively censored samples, competing-risks models, recurrent-event data, and multistate processes. From an inferential perspective, the development of asymptotic theory, goodness-of-fit procedures, and hypothesis tests based on the proposed Kolmogorov and Wasserstein measures deserves further attention. More generally, the probabilistic framework introduced here may provide a useful foundation for studying independence properties involving other order statistics, rank indicators, and related stochastic processes.

\section*{Declarations}

\paragraph*{Ethics Approval}
Not applicable.

 \paragraph*{Funding Declaration}
 This study was financed in part by CAPES (Finance Code 001).


\paragraph*{Disclosure statement}
There are no conflicts of interest to disclose.

\paragraph*{Author Contributions Statement}
All authors contributed equally to the conception and design of the study, data analysis, interpretation of the results, manuscript preparation, and revision. All authors read and approved the final version of the manuscript.

\paragraph*{Data availability}
No data were used in this study.

\bibliographystyle{apalike}
\bibliography{refbib}

\end{document}